\documentclass[12pt]{amsart}
\usepackage{amscd, amsfonts, amssymb, amsthm}
\usepackage{parskip, fullpage, verbatim}
\usepackage[colorlinks, breaklinks, linkcolor=blue]{hyperref} 
\usepackage{breakurl}
\usepackage{array}
\usepackage{booktabs}
\usepackage{wasysym} % for /lightning

\newcommand\CA{{\mathcal A}} 
\newcommand\CB{{\mathcal B}}
\newcommand\CC{{\mathcal C}}

\newcommand\CIF{{\mathcal {IF}}} 
\newcommand\HIF{{\mathcal {HIF}}}

\newcommand\oS{{\overline S}}

\newcommand\BBC{{\mathbb C}}

\newcommand\BBZ{{\mathbb Z}}

\newcommand {\GAP}{\textsf{GAP}}  
\newcommand {\CHEVIE}{\textsf{CHEVIE}}  
\newcommand {\Singular}{\textsf{SINGULAR}}  
\newcommand {\Sage}{\textsf{SAGE}}

\newcommand\Der{{\operatorname{Der}}}

\newcommand\GL{\operatorname{GL}}

\newcommand\pdeg{\operatorname{pdeg}}

\numberwithin{equation}{section}

\theoremstyle{plain}
\newtheorem{lemma}[equation]{Lemma}
\newtheorem{theorem}[equation]{Theorem}

\newtheorem{corollary}[equation]{Corollary}
\newtheorem{proposition}[equation]{Proposition}
\theoremstyle{definition}
\newtheorem{defn}[equation]{Definition}
\newtheorem{remark}[equation]{Remark}
\newtheorem{example}[equation]{Example}

\subjclass[2010]{Primary 20F55, 52B30, 52C35, 14N20; Secondary 13N15}

\begin{document}

%%%%%%%%%%%%%%%%%%%%%%%%%%%%%%%%%%%%%%%%%%%%%%%%%%%%%%%%%%%%%%%%%%%%%%
%%%%%%%%%%%%% top matter stuff
%%%%%%%%%%%%%%%%%%%%%%%%%%%%%%%%%%%%%%%%%%%%%%%%%%%%%%%%%%%%%%%%%%%%%%
\title[On inductively free reflection arrangements]
{on inductively free reflection arrangements}

%\dedicatory{preliminary version}

%%% first author info
\author[T. Hoge]{Torsten Hoge}
\address
{Fakult\"at f\"ur Mathematik,
Ruhr-Universit\"at Bochum,
D-44780 Bochum, Germany}
\email{torsten.hoge@rub.de}

%%% second author info
\author[G. R\"ohrle]{Gerhard R\"ohrle}
\address
{Fakult\"at f\"ur Mathematik,
Ruhr-Universit\"at Bochum,
D-44780 Bochum, Germany}
\email{gerhard.roehrle@rub.de}

%%\date{\today}
\keywords{Complex reflection groups,
reflection arrangements, free arrangements, 
inductively free arrangements, hereditarily inductively free arrangements}

\allowdisplaybreaks

\begin{abstract}
Suppose that $W$ is a finite, unitary 
reflection group acting on the complex 
vector space $V$.
Let $\CA = \CA(W)$ be the associated 
hyperplane arrangement of $W$.
Terao \cite{terao:freeI} has shown that 
each such reflection
arrangement $\CA$ is free.
There is the stronger notion of an inductively free arrangement. 
In 1992, Orlik and Terao \cite[Conj.\ 6.91]{orlikterao:arrangements}
conjectured that each reflection arrangement is inductively free.
It has been  known for quite some time that 
the braid arrangement %$\CA(S_\ell)$ 
as well as the 
Coxeter arrangements of type $B_\ell$
and type $D_\ell$
are inductively free. 
Barakat and Cuntz \cite{cuntz:indfree} 
completed this list only recently by showing 
that every Coxeter arrangement is inductively free.

Nevertheless, Orlik and Terao's conjecture is false in general.
In \cite{hogeroehrle:free}, we already gave
two counterexamples to this conjecture among the exceptional
complex reflection groups.

In this paper we
classify all inductively free reflection arrangements.
In addition, we show that the notions of inductive freeness
and that of hereditary inductive freeness
coincide for reflection arrangements.

As a consequence of our classification, we get an
easy, purely combinatorial characterization of 
inductively free %and hereditarily inductively free
reflection arrangements $\CA$ 
in terms of exponents of the restrictions 
to any hyperplane of $\CA$.
\end{abstract}

\maketitle

%%%%%%%%%%%%%%%%%%%%%%%%%%%%%%%%%%%%%%%%%%%%%%%%%%%%%%%%%%%%%%%%%%%%%%
%%%%%%%%%%%%% article body...
%%%%%%%%%%%%%%%%%%%%%%%%%%%%%%%%%%%%%%%%%%%%%%%%%%%%%%%%%%%%%%%%%%%%%%

%%%%%%%%%%%%%%%%%%%%%%%%%%%%%%%%%%%%%%%%%%%%%%%%%%%%%%%%%%%%%%%%%%%%%%
%%%%%%%%%%%%% \S1 Introduction
%%%%%%%%%%%%%%%%%%%%%%%%%%%%%%%%%%%%%%%%%%%%%%%%%%%%%%%%%%%%%%%%%%%%%%
\section{Introduction}

Suppose that $W$ is a finite, unitary
reflection group acting on the complex 
vector space $V$.
Let $\CA = (\CA(W),V)$ be the associated 
hyperplane arrangement of $W$.
Terao \cite{terao:freeI} has shown that each reflection
arrangement $\CA$ is free and that 
the multiset of exponents 
$\exp \CA$ of $\CA$ is given by the 
coexponents %$\coexp(W)$ 
of $W$;
see also \cite[\S 6]{orlikterao:arrangements}.

For $\CA$ an arrangement
let $L(\CA)$ be the intersection lattice of $\CA$.
For a subspace $X$ in $L(\CA)$ we have the restricted
arrangement $\CA^X$ in $X$ by means of restricting hyperplanes from $\CA$
to $X$.
In 1992, Orlik and Terao \cite[Conj.\ 6.90]{orlikterao:arrangements}
conjectured that each such restriction is again free
in case $\CA$ is a reflection arrangement.
Free arrangements with this property are called \emph{hereditarily free},
 \cite[Def.\ 4.140]{orlikterao:arrangements}. 
All but a few cases of this conjecture were settled in 
\cite{orlikterao:arrangements} and \cite{orlikterao:free};
recently, we resolved the outstanding cases 
in \cite{hogeroehrle:free},
confirming the conjecture.

There are various stronger notions of freeness,
in particular that of
an \emph{inductively free} arrangement 
due to Terao, \cite{terao:freeI};
see Definition \ref{def:indfree}.
If $\CA^X$ is inductively free for each $X \in L(\CA)$,
then $\CA$ is called   \emph{hereditarily inductively free},
cf.\ \cite[\S 6.4, p.~253]{orlikterao:arrangements}.
In 1992, Orlik and Terao \cite[Conj.\ 6.91]{orlikterao:arrangements}
conjectured that each reflection arrangement is hereditarily  inductively free.
Recently, Barakat and Cuntz \cite[Cor.\ 5.15]{cuntz:indfree} showed that  
all \emph{crystallographic} arrangement are indeed 
hereditarily inductively free. 
These include all  
Coxeter arrangements whose underlying Coxeter group is 
crystallographic, i.e.\ all Weyl groups.
The most challenging case here is that of 
the Coxeter group of type $E_8$.
If $\CA$ is hereditarily inductively free, then it is 
inductively free as well. 

While it has been known for quite some time that 
the braid arrangement $\CA(S_\ell)$ 
as well as the 
Coxeter arrangements of type $B_\ell$
%for $\ell \ge 2$  
and type $D_\ell$
%for $\ell \ge 4$ 
are inductively free, 
see \cite[Ex.\ 4.55]{orlikterao:arrangements} and 
\cite[Ex.\ 2.6]{jamboterao:free}, 
it was only very recently that 
Barakat and Cuntz \cite{cuntz:indfree} 
completed this list
by showing that every Coxeter arrangement is inductively free
(including all Coxeter groups of exceptional type).
Nevertheless, Orlik and Terao's conjecture  \cite[Conj.\ 6.91]{orlikterao:arrangements} 
mentioned above is false in general;
in \cite{hogeroehrle:free}, we already gave two counterexamples,
namely the reflection arrangements of $G_{33}$ and $G_{34}$ 
are not inductively free.

So it is natural to determine the class of 
inductively free reflection arrangements and 
to characterize it, ideally in a combinatorial fashion.
These are the goals of this paper. 

Firstly, we classify all inductively free reflection arrangements,
see Theorem \ref{main}.
Secondly, we show that the notions of inductive freeness
and that of hereditarily inductive freeness
coincide for reflection arrangements,
see Theorem \ref{main2}.
This equivalence is rather surprising,
as the underlying classes of free arrangements are distinct as such, %
see Example \ref{ex:hindf-not-indf}.

Finally, 
as a consequence of our classification
in Theorem \ref{main}, we obtain an 
easy, purely combinatorial characterization of 
inductively free 
reflection arrangements $\CA(W)$ 
in terms of the exponents of the 
restriction of $\CA(W)$ to any hyperplane,
see Corollary \ref{cor:main}.

The relevance of free arrangements $\CA$ lies 
in the fact that they satisfy the so called factorization property 
of the Poincar\'e polynomial of its lattice $L(\CA)$,
cf.\ \cite[\S 2.3]{orlikterao:arrangements}.
If $\CA$ is free, then the Poincar\'e polynomial 
$\pi(\CA,t)$ of $L(\CA)$
factors into linear terms as follows:
\[
\pi(\CA,t) = \prod_{i=1}^\ell (1 + b_i t)
\]
where $\exp \CA = \{b_1, \ldots, b_\ell\}$ are the exponents of $\CA$,
\cite[Thm.\ 4.137]{orlikterao:arrangements}.
Since $\pi(\CA,t)$ is defined only in terms of the M\"obius function 
of $L(\CA)$, this factorization property suggests
that freeness of $\CA$ only depends on the 
lattice $L(\CA)$; indeed, this assertion 
is a fundamental conjecture due to Terao, 
\cite[Conj.\ 4.138]{orlikterao:arrangements}.

Here is our principal result
where we use the classification and 
labelling of the irreducible 
unitary reflection groups due to
Shephard and Todd, \cite{shephardtodd}. 

\begin{theorem}
\label{main}
For $W$ a finite complex reflection group,  
the reflection arrangement $\CA(W)$ of $W$ is 
inductively free if and only if 
$W$ does not admit an irreducible factor
isomorphic to a monomial group 
$G(r,r,\ell)$ for $r, \ell \ge 3$, 
$G_{24}, G_{27}, G_{29}, G_{31}, G_{33}$, or $G_{34}$.
\end{theorem}

As indicated above, the case for Coxeter groups 
was only recently established in \cite{cuntz:indfree}.

Our second main result shows that the notions of 
inductive freeness and that of hereditary inductive freeness coincide
for reflection arrangements.
This is 
rather surprising, as these two classes of free arrangements 
differ as such, cf.\ Example \ref{ex:hindf-not-indf}.

\begin{theorem}
\label{main2}
For $W$ a finite complex reflection group,  
let $\CA = \CA(W)$ be its reflection arrangement.
Then $\CA$ is inductively free
if and only if  $\CA$ is hereditarily inductively free.
\end{theorem}

Strikingly,
as a consequence of our classification, we obtain an  
easy, purely combinatorial characterization of 
inductively free 
reflection arrangements in terms of exponents.

\begin{corollary}
\label{cor:main}
For $W$ a finite 
complex reflection group,  
let $\CA = \CA(W)$ be the reflection arrangement of $W$.
Suppose that $W$ does not admit an irreducible factor
isomorphic to $G_{31}$.
Then $\CA$ is inductively free if and only if 
$\exp \CA^H \subseteq \exp \CA$ for any $H \in \CA$.%
\end{corollary}

Corollary \ref{cor:main} follows immediately from 
Theorem \ref{main} and the explicit lists of 
exponents $\exp A^H$ of the restrictions $\CA^H$ 
of $\CA$ to any hyperplane $H$ in $\CA$ 
from \cite[\S 6, App.\ C]{orlikterao:arrangements}.
Note that by \cite{hogeroehrle:free},  
$\CA^H$ is known to be free also for $G_{33}$ and $G_{34}$
with exponents given as in \cite[Tables C.14, C.17]{orlikterao:arrangements}.

The paper is organized as follows.
In the next section we recall the required notation and  
facts about freeness of hyperplane arrangements,
inductively free arrangements and 
reflection arrangements, mostly taken from
\cite[\S 4, \S6]{orlikterao:arrangements}.

In Proposition \ref{prop:product-indfree}
we show that inductively free arrangements 
are compatible with the product 
construction for arrangements,
and extend this to 
hereditarily inductively free arrangements 
in Corollary \ref{cor:product-heredindfree}.
In Lemma \ref{lem:3-arr}, we observe that 
a $3$-arrangement is 
inductively free if and only if
it is hereditarily inductively free.

Our key criterion, 
Corollary \ref{cor:nonindfree}, 
shows that the reflection arrangement 
$\CA = \CA(W)$ is not inductively free
provided 
$\exp \CA^H \not\subseteq \exp \CA$, 
for any restriction $\CA^H$ of $\CA$.

Theorems \ref{main} and  \ref{main2}
are proved in Section \ref{sect:proof}.
In Propositions \ref{prop:gr1l} and \ref{prop:grrl}, we show that the 
arrangements of the
monomial groups $G(r,p,\ell)$ for $p \ne r$ and $\ell \ge 2$ 
are always inductively free while those of the  
monomial groups  $G(r,r,\ell)$ for $r, \ell \ge 3$ 
are not.

For general information about arrangements and reflection groups we refer
the reader to \cite{bourbaki:groupes} and \cite{orlikterao:arrangements}.

\section{Recollections and Preliminaries}

\subsection{Hyperplane Arrangements}
\label{ssect:hyper}

Let $V = \BBC^\ell$ 
be an $\ell$-dimensional complex vector space.
A \emph{hyperplane arrangement} is a pair
$(\CA, V)$, where $\CA$ is a finite collection of hyperplanes in $V$.
Usually, we simply write $\CA$ in place of $(\CA, V)$.
We only consider central arrangements.
We write $n = |\CA|$
for the number of hyperplanes in $\CA$.
The empty arrangement in $V$ is denoted by $\Phi_\ell$.

The \emph{lattice} $L(\CA)$ of $\CA$ is the set of subspaces of $V$ of
the form $H_1\cap \dotsm \cap H_r$ where $\{ H_1, \ldots, H_r \}$ is a subset
of $\CA$. 
For $X \in L(\CA)$, we have two associated arrangements, 
firstly the subarrangement 
$\CA_X :=\{H \in \CA \mid X \subseteq H\} \subseteq \CA$
of $\CA$ and secondly, 
the \emph{restriction of $\CA$ to $X$}, $(\CA^X,X)$, where 
$\CA^X := \{ X \cap H \mid H \in \CA \setminus \CA_X\}$.
Note that $V$ belongs to $L(\CA)$
as the intersection of the empty 
collection of hyperplanes and $\CA^V = \CA$. 

For $\CA \ne \Phi_\ell$, 
let $H_0 \in \CA$.
Define $\CA' := \CA \setminus\{ H_0\}$,
and $\CA'' := \CA^{H_0} = \{ H_0 \cap H \mid H \in \CA'\}$.
Then $(\CA, \CA', \CA'')$ is a \emph{triple} of arrangements,
\cite[Def.\ 1.14]{orlikterao:arrangements}.

The \emph{product}
$\CA = (\CA_1 \times \CA_2, V_1 \oplus V_2)$ 
of two arrangements $(\CA_1, V_1), (\CA_2, V_2)$
is defined by
\begin{equation}
\label{eq:product}
\CA := \CA_1 \times \CA_2 = \{H_1 \oplus V_2 \mid H_1 \in \CA_1\} \cup 
\{V_1 \oplus H_2 \mid H_2 \in \CA_2\},
\end{equation}
see \cite[Def.\ 2.13]{orlikterao:arrangements}.
In particular,  
$|\CA| = |\CA_1| + |\CA_2|$. 

Note that 
$\CA \times \Phi_0 = \CA$
for any arrangement $\CA$. 
If $\CA$ is of the form $\CA = \CA_1 \times \CA_2$, where 
$\CA_i \ne \Phi_0$ for $i=1,2$, then $\CA$
is called \emph{reducible}, else $\CA$
is %called 
\emph{irreducible}, 
\cite[Def.\ 2.15]{orlikterao:arrangements}.

For instance, the braid arrangement $\CA(S_\ell)$ is 
the product of the empty $1$-arrangement and an irreducible arrangement, 
\cite[Ex.\ 2.16]{orlikterao:arrangements}.

Let $\CA = \CA_1 \times \CA_2$ be a product. 
By  \cite[Prop.\ 2.14]{orlikterao:arrangements},
there is a lattice isomorphism
\[
 L(\CA_1) \times L(\CA_2) \cong L(\CA) \quad \text{by} \quad
(X_1, X_2) \mapsto X_1 \oplus X_2.
\]
Using \eqref{eq:product}, it is easily seen that
for $X =  X_1 \oplus X_2 \in L(\CA)$, we have 
$\CA _X =  ({\CA_1})_{X_1} \times ({\CA_2})_{X_2}$
and 
\begin{equation}
\label{eq:restrproduct}
\CA^X = \CA_1^{X_1} \times \CA_2^{X_2}.
\end{equation}

\subsection{Free Arrangements}
\label{ssect:free}

Let $S = S(V^*)$ be the symmetric algebra of the dual space $V^*$ of $V$.
If $x_1, \ldots , x_\ell$ is a basis of $V^*$, then we identify $S$ with 
the polynomial ring $\BBC[x_1, \ldots , x_\ell]$.
Letting $S_p$ denote the $\BBC$-subspace of $S$
consisting of the homogeneous polynomials of degree $p$ (along with $0$),
we see that
$S$ is naturally $\BBZ$-graded: $S = \oplus_{p \in \BBZ}S_p$, where
$S_p = 0$ for $p < 0$.

Let $\Der(S)$ be the $S$-module of $\BBC$-derivations of $S$.
For $i = 1, \ldots, \ell$, 
let $D_i := \partial/\partial x_i$ be the usual derivation of $S$.
Then $D_1, \ldots, D_\ell$ is an $S$-basis of $\Der(S)$.
We say that $\theta \in \Der(S)$ is 
\emph{homogeneous of polynomial degree p}
provided 
$\theta = \sum_{i=1}^\ell f_i D_i$, 
where $f_i \in S_p$ for each $1 \le i \le \ell$.
In this case we write $\pdeg \theta = p$.
Let $\Der(S)_p$ be the $\BBC$-subspace of $\Der(S)$ consisting 
of all homogeneous derivations of polynomial degree $p$.
Then $\Der(S)$ is a graded $S$-module:
$\Der(S) = \oplus_{p\in \BBZ} \Der(S)_p$.

Following \cite[Def.~4.4]{orlikterao:arrangements}, 
for $f \in S$, we define the $S$-submodule $D(f)$ of $\Der(S)$ by
\[
D(f) := \{\theta \in \Der(S) \mid \theta(f) \in f S\} .
\]

Let $\CA$ be an arrangement in $V$. 
Then for $H \in \CA$ we fix $\alpha_H \in V^*$ with
$H = \ker \alpha_H$.
The \emph{defining polynomial} $Q(\CA)$ of $\CA$ is given by 
$Q(\CA) := \prod_{H \in \CA} \alpha_H \in S$.

The \emph{module of $\CA$-derivations} of $\CA$ is 
defined by 
\[
D(\CA) := D(Q(\CA)).
\]
We say that $\CA$ is \emph{free} if the module of $\CA$-derivations
$D(\CA)$ is a free $S$-module.
The notion of freeness was introduced by Saito in his 
seminal work \cite{saito}.

With the $\BBZ$-grading of $\Der(S)$, the module of $\CA$-derivations
becomes a graded $S$-module $D(\CA) = \oplus_{p\in \BBZ} D(\CA)_p$,
where $D(\CA)_p = D(\CA) \cap \Der(S)_p$, 
\cite[Prop.\ 4.10]{orlikterao:arrangements}.
If $\CA$ is a free arrangement, then the $S$-module %of $\CA$-derivations
$D(\CA)$ admits a basis of $\ell$ homogeneous derivations, 
say $\theta_1, \ldots, \theta_\ell$, \cite[Prop.\ 4.18]{orlikterao:arrangements}.
While the $\theta_i$'s are not unique, their polynomial 
degrees $\pdeg \theta_i$ 
are unique (up to ordering). This multiset is the set of 
\emph{exponents} of the free arrangement $\CA$
and is denoted by $\exp \CA$.

The so called \emph{Addition-Deletion Theorem} 
due to Terao  \cite{terao:freeI} plays a 
crucial role in the study of free arrangements, 
\cite[Thm.\ 4.51]{orlikterao:arrangements}.

\begin{theorem}
\label{thm:add-del}
Suppose that $\CA \ne \Phi_\ell$.
Let  $(\CA, \CA', \CA'')$ be a triple of arrangements. Then any 
two of the following statements imply the third:
\begin{itemize}
\item $\CA$ is free with $\exp \CA = \{ b_1, \ldots , b_{\ell -1}, b_\ell\}$;
\item $\CA'$ is free with $\exp \CA' = \{ b_1, \ldots , b_{\ell -1}, b_\ell-1\}$;
\item $\CA''$ is free with $\exp \CA'' = \{ b_1, \ldots , b_{\ell -1}\}$.
\end{itemize}
\end{theorem}

Owing to \cite[Prop.\ 4.28]{orlikterao:arrangements}, 
free arrangements behave well with respect to 
the  product construction for arrangements.

\begin{proposition}
\label{prop:product-free}
Let $\CA_1, \CA_2$ be two arrangements.
Then  $\CA = \CA_1 \times \CA_2$ is free
if and only if both 
$\CA_1$ and $\CA_2$ are free and in that case
the multiset of exponents of $\CA$ is given by 
$\exp \CA = \{\exp \CA_1, \exp \CA_2\}$.
\end{proposition}

Suppose that  $\CA \ne \Phi_\ell$. 
Fix $H_0 \in \CA$ and consider
the triple $(\CA,\CA', \CA'')$  
associated to $H_0$. Furthermore, let 
$\alpha_0 \in V^*$ so that $H_0 = \ker \alpha_0$.
Following \cite[Def.\ 4.43]{orlikterao:arrangements},
set $\oS := S/\alpha_0 S$. Since, for $\theta \in D(\CA)$
we have $\theta(\alpha_0 S) \subseteq \alpha_0 S$,
we may define $\overline{\theta} : \oS \to \oS$ by
$\overline{\theta}(f + \alpha_0 S) = \theta(f) + \alpha_0 S$.
We recall some properties of 
this construction, \cite[Prop.\ 4.44, Prop.\ 4.57]{orlikterao:arrangements}.

\begin{proposition}
\label{prop:q}
If $\theta \in D(\CA)$, then $\overline{\theta} \in D(\CA'')$.
Moreover, if $\overline{\theta} \ne 0$, then 
$\pdeg \overline{\theta} = \pdeg \theta$.
\end{proposition}

Thanks to Proposition \ref{prop:q}, we obtain a ``degree-preserving'' map
\[
q : D(\CA) \to D(\CA'') \ \text{ by }\ \theta \mapsto \overline{\theta}.
\]

\begin{proposition}
\label{prop:qonto}
Suppose that both $\CA$ and $\CA''$ are free.
Then $q : D(\CA) \to D(\CA'')$ is onto if and only if 
$\CA'$ is free.
\end{proposition}

We obtain the following immediate 
consequence of  
Propositions \ref{prop:q} and \ref{prop:qonto}.

\begin{corollary}
\label{cor:q}
Suppose that $\CA$ and $\CA''$ are free
and $\exp \CA'' \not\subseteq \exp \CA$.
Then $\CA'$ is not free.
\end{corollary}

\begin{proof}
Since $\exp \CA'' \not\subseteq \exp \CA$, it 
follows from the second assertion of 
Proposition \ref{prop:q} that $q$ is not onto.
Consequently, $\CA'$ is not free,
by Proposition \ref{prop:qonto}.
\end{proof}

\subsection{Inductively Free Arrangements}
\label{ssect:indfree}

Theorem \ref{thm:add-del} motivates the notion of 
\emph{inductively free} arrangements,   
\cite[Def.\ 4.53]{orlikterao:arrangements}.

\begin{defn}
\label{def:indfree}
The class $\CIF$ of \emph{inductively free} arrangements 
is the smallest class of arrangements subject to
\begin{itemize}
\item[(i)] $\Phi_\ell \in \CIF$ for each $\ell \ge 0$;
\item[(ii)] if there exists a hyperplane $H_0 \in \CA$ such that both
$\CA'$ and $\CA''$ belong to $\CIF$, and $\exp \CA '' \subseteq \exp \CA'$, 
then $\CA$ also belongs to $\CIF$.
\end{itemize}
\end{defn}

\begin{remark}
\label{rem:indtable}
It is possible to describe an inductively free arrangement $\CA$ by means of 
a so called 
\emph{induction table}, cf.~\cite[\S 4.3, p.~119]{orlikterao:arrangements}.
In this process we start with an inductively free arrangement
(frequently $\Phi_\ell$) and add hyperplanes successively ensuring that 
part (ii) of Definition \ref{def:indfree} is satisfied.
This process is referred to as \emph{induction of hyperplanes}.
This procedure amounts to 
choosing a total order on $\CA$, say 
$\CA = \{H_1, \ldots, H_n\}$, 
so that each of the subarrangements 
$\CA_0 := \Phi_\ell$, $\CA_i := \{H_1, \ldots, H_i\}$
and each of the restrictions $\CA_i^{H_i}$ is inductively free
for $i = 1, \ldots, n$.
In the associated induction table we record in the $i$-th row the information 
of the $i$-th step of this process, by 
listing $\exp \CA_i' = \exp \CA_{i-1}$, 
the defining form $\alpha_{H_i}$ of $H_i$, 
as well as $\exp \CA_i'' = \exp \CA_i^{H_i}$, 
for $i = 1, \ldots, n$.
For instance, see Tables \ref{indtable:non-HIF}, 
\ref{indtable1}, \ref{indtable:g26}, and \ref{indtableg32} below. 
\end{remark}

Our next result shows that
the compatibility of products and free arrangements
from Proposition \ref{prop:product-free}  
restricts to the class of 
inductively free arrangements. 

\begin{proposition}
\label{prop:product-indfree}
Let $(\CA_1, V_1), (\CA_2, V_2)$ be two arrangements.
Then  $\CA = (\CA_1 \times \CA_2, V_1 \oplus V_2)$ is 
inductively free if and only if both 
$(\CA_1, V_1)$ and $(\CA_2, V_2)$ are 
inductively free and in that case
the multiset of exponents of $\CA$ is given by 
$\exp \CA = \{\exp \CA_1, \exp \CA_2\}$.
\end{proposition}

\begin{proof}
Let $H = H_1  \oplus V_2 \in \CA$ with $H_1 \in \CA_1$,
cf.\ \eqref{eq:product}.
Then
$\CA^H = \CA_1^{H_1} \times \CA_2^{V_2} = \CA_1^{H_1} \times \CA_2$,
thanks to \eqref{eq:restrproduct}.
Likewise, for 
$H = V_1  \oplus H_2 \in \CA$ with $H_2 \in \CA_2$, 
we have $\CA^H =  \CA_1 \times \CA_2^{H_2}$.

Moreover, 
for $H = H_1  \oplus V_2 \in \CA$ with $H_1 \in \CA_1$,
we have $\CA\setminus \{H\} = (\CA_1\setminus \{H_1\}) \times \CA_2$
and likewise for $H = V_1  \oplus H_2 \in \CA$ with $H_2 \in \CA_2$, 
we have $\CA\setminus \{H\} =  \CA_1 \times (\CA_2\setminus \{H_2\})$.

First suppose that both 
$\CA_1$ and $\CA_2$ are inductively free. 
We show that $\CA = \CA_1 \times \CA_2$ is inductively free
by induction on $n = |\CA| = |\CA_1| + |\CA_2|$. For 
$n = 0$ we have $\CA = \Phi_\ell$ and there is nothing to prove.
Now suppose that $n \ge 1$.
So we may assume that there are $H_i \in \CA_i$ so that 
$\CA_i \setminus \{H_i\}$ and 
$\CA_i^{H_i}$ are  inductively free and that 
$\exp \CA_i^{H_i} \subseteq \exp (\CA_i \setminus \{H_i\})$ 
 for $i = 1,2$ or else one of the $\CA_i$ is empty.

Without loss, assume that  $H = H_1  \oplus V_2 \in \CA$ with $H_1 \in \CA_1$,
so that $\CA^H = \CA_1^{H_1} \times \CA_2$. 
(The case when $H = V_1  \oplus H_2 \in \CA$ with $H_2 \in \CA_2$
is treated in a similar way.)
Then, since
$\CA^H = \CA_1^{H_1} \times \CA_2$ and $|\CA^H| < n$,
it follows from 
our induction hypothesis and the assumptions on $\CA_1$ and $\CA_2$
that $\CA^H$ is inductively free 
and that 
$\exp \CA^H = \{\exp  \CA_1^{H_1},  \exp \CA_2\}$.

Further, since 
$\CA\setminus \{H\} = \CA_1\setminus \{H_1\} \times \CA_2$ and $|\CA\setminus \{H\}| < n$, 
we conclude by our induction hypothesis and the assumptions on $\CA_1$ and $\CA_2$ that 
$\CA\setminus \{H\}$ is inductively free and that 
$\exp (\CA\setminus \{H\}) = \{ \exp(\CA_1\setminus \{H_1\}), \exp \CA_2\}$.
Since $ \exp  \CA_1^{H_1} \subseteq \exp(\CA_1\setminus \{H_1\})$, 
we obtain
\[
\exp \CA^H= \{ \exp  \CA_1^{H_1},  \exp \CA_2 \}  \subseteq  
\{ \exp(\CA_1\setminus \{H_1\}), \exp \CA_2\} = \exp (\CA\setminus \{H\}).  
\]
Consequently, $\CA$ is inductively free, as claimed.

Conversely, suppose that $\CA$ is inductively free. We show that both 
$\CA_1$ and $\CA_2$ are inductively free again by induction on $n = |\CA|$.

If $n = 0$, then $\CA = \Phi_\ell$ and so both $\CA_1$ and $\CA_2$
are empty and there is nothing to show.
So suppose that $n \ge 1$. Since $\CA$ is inductively free, 
there is a hyperplane $H$ in $\CA$, so that 
$(\CA, \CA\setminus \{H\}, \CA^H)$ is a triple of inductively free
arrangements with 
$\exp \CA^H \subseteq \exp \CA\setminus \{H\}$.
Without loss, we may assume that $H$ is of the form 
$H = H_1 \oplus V_2$ for some $H_1 \in \CA_1$.
Then $\CA\setminus \{H\} = (\CA_1\setminus \{H_1\}) \times \CA_2$ 
and $\CA^H = \CA_1^{H_1} \times \CA_2$. 

Since $|\CA\setminus \{H\}| < n$ and $|\CA^{H}| < n$,
it follows from our induction hypothesis,
the fact that both $\CA\setminus \{H\}$ and $\CA^{H}$
are  products and the 
assumption that both $\CA\setminus \{H\}$ and $\CA^H$ are inductively free,
that $\CA_1\setminus \{H_1\}$, $\CA_1^{H_1}$ and $\CA_2$
are inductively free.

Since $\exp \CA^H \subseteq \exp \CA\setminus \{H\}$, we get 
\[
\{\exp  \CA_1^{H_1}, \exp \CA_2 \} \subseteq 
\{ \exp (\CA_1\setminus \{H_1\}), \exp \CA_2 \}, 
\]
and since this is a containment of multisets, we can conclude that 
\[
\exp  \CA_1^{H_1} \subseteq \exp (\CA_1\setminus \{H_1\}).
\]
Thus, $\CA_1$ satisfies Definition \ref{def:indfree}(ii), so
$\CA_1$ is also inductively free.

The final statement on exponents  follows from 
Proposition \ref{prop:product-free}.
\end{proof}

There is yet an even stronger notion of freeness,
cf.\  \cite[\S 6.4, p.~253]{orlikterao:arrangements}.

\begin{defn}
\label{def:heredindfree}
The arrangement $\CA$ is called 
\emph{hereditarily inductively free} provided 
that $\CA^X$ is inductively free for each $X \in L(\CA)$.
We sometimes abbreviate this class by $\HIF$.
\end{defn}

Note, 
as $V \in L(\CA)$ and $\CA^V = \CA$, 
$\CA$ is inductively free, if it is 
hereditarily inductively free.

For instance, the empty arrangement $\Phi_\ell$ is
vacuously  hereditarily inductively free.

Using \eqref{eq:restrproduct} and Proposition \ref{prop:product-indfree},
Proposition \ref{prop:product-free}
restricts to the class of hereditarily  inductively free arrangements.

\begin{corollary}
\label{cor:product-heredindfree}
Let $\CA_1$ and $\CA_2$ be two arrangements.
Then  $\CA = \CA_1 \times \CA_2$ is 
hereditarily inductively free if and only if both
$\CA_1$ and $\CA_2$   are 
hereditarily inductively free and in that case
the multiset of exponents of $\CA$ is given by 
$\exp \CA = \{\exp \CA_1, \exp \CA_2\}$.
\end{corollary}

\begin{proof}
First suppose that both $\CA_1$ and $\CA_2$ are
 hereditarily  inductively free.
Let $X = X_1 \oplus X_2$ be in $L(\CA)$.
Then, by \eqref{eq:restrproduct} and Proposition \ref{prop:product-indfree},
$\CA^X = \CA_1^{X_1} \times \CA_2^{X_2}$ is 
inductively free.

Conversely, suppose that 
$\CA \in \HIF$. 
Let $X_i \in L(\CA_i)$ for $i=1,2$. %and $X_2 \in L(\CA_2)$.
Then $X = X_1 \oplus X_2 \in L(\CA)$.
By \eqref{eq:restrproduct} and Proposition \ref{prop:product-indfree},
both $\CA_1^{X_1}$ and $\CA_2^{X_2}$ are
inductively free. 

The final statement on exponents  follows from Proposition \ref{prop:product-free}.
\end{proof}

Owing to \cite[Def.\ 4.7; Prop.\ 4.27]{orlikterao:arrangements}
and \cite[Ex.\ 4.20]{orlikterao:arrangements},
all $1$- and $2$-arrangements are known to be free.
Next we observe that they are also always 
hereditarily  inductively free.

\begin{example}
\label{ex:1-arr}
Any $1$-arrangement $\CA$ is hereditarily 
inductively free.
If $\CA = \Phi_1$, there is nothing to prove.
So let $\CA = (\{0\}, \BBC)$.
Then 
$(\CA, \CA', \CA'') = (\CA, \Phi_1, \Phi_0)$
is a triple 
with $\CA', \CA'' \in \CIF$ and
$\exp \CA'' = \varnothing \subset \{0\} = \exp \CA'$.  
So, by Definition \ref{def:indfree}, $\CA$ is 
inductively free with $\exp \CA =\{1\}$.
Since $\CA^{\{0\}} = \Phi_0$,
we get $\CA \in \HIF$.
\end{example}

\begin{lemma}
\label{lem:2-arr}
Any $2$-arrangement is hereditarily inductively free.
\end{lemma}

\begin{proof}
Let $\CA$ be a $2$-arrangement.
First we show that $\CA$ is inductively free 
by induction on $n = |\CA|$.
If $n=0$, then $\CA = \Phi_2$ and there is nothing to prove.
Suppose that $n \ge 1$.
If $ n = 1$, then 
$(\CA, \CA', \CA'') = (\CA, \Phi_2, \Phi_1)$
is a triple 
with $\CA', \CA'' \in \CIF$ and
 $\exp \CA'' = \{0\} \subset \{0, 0\} = \exp \CA'$.  
So $\CA$ is 
inductively free with $\exp \CA =\{0, 1\}$.  
Now let $n \ge 2$.
By our inductive hypothesis, $\CA'$ is inductively free
with $\exp \CA' =\{1, n -2\}$.  
By Example \ref{ex:1-arr}, $\CA''$ is inductively free
with $\exp \CA'' =\{1\} \subset \exp \CA'$.  
Thus, by Definition \ref{def:indfree},
$\CA$ is inductively free with $\exp \CA =\{1, n -1\}$.

By Example \ref{ex:1-arr},
$\CA^H$ is inductively free for any $H \in \CA$.
If $H, H'$ are distinct hyperplanes in $\CA$, then $X = H \cap H' = \{0\}$
and so $\CA^X = \Phi_0$. % is also inductively free.
Consequently, $\CA \in \HIF$. % is hereditarily inductively free.
\end{proof}

In general, a free $3$-arrangement need not be inductively free,
see \cite[Ex.\ 4.59]{orlikterao:arrangements};
not even if it is a  reflection arrangement,
see Example \ref{ex:g333} below.
Nevertheless, for a $3$-arrangement, 
the two stronger notions of inductive freeness coincide.

\begin{lemma}
\label{lem:3-arr}
Suppose that $\ell = 3$. Then 
$\CA \in \CIF$ if and only if $\CA \in \HIF$.
\end{lemma}

\begin{proof}
The reverse implication is clear.
So assume that $\CA \in \CIF$. % is inductively free.
Let $V \ne X \in L(\CA)$.
Then $\CA^X$ is a $d$-arrangement  for $d \le 2$ and 
so $\CA^X \in \CIF$, %is inductively free, 
by Example \ref{ex:1-arr} and Lemma \ref{lem:2-arr}.
\end{proof}

Our next example shows that
$\HIF$ is a proper subclass of $\CIF$.
By Lemma \ref{lem:3-arr}, any such example 
can only occur in dimension at least $4$.

\begin{example}
\label{ex:hindf-not-indf}
Let $\CA$ be the $4$-arrangement defined by the 
$10$ forms $\alpha_H$ shown in 
column two of Table \ref{indtable:non-HIF},
where we denote the coordinate functions in $S$ simply 
by $a, b, c$ and $d$.
We claim that $\CA$ is inductively free but not hereditarily inductively 
free. 

That $\CA$ is inductively free follows from the data in 
the induction table of $\CA$  %presented 
in Table \ref{indtable:non-HIF}
below along with the fact that
each occurring restriction $\CA''$ is itself 
again inductively free with the given set of exponents. 
We have checked this latter condition directly.
We omit the details.

\begin{table}[h]
\begin{tabular}{lll}  \hline
  $\exp \CA'$ & $\alpha_H$ & $\exp \CA''$\\
%    [5pt]  
\hline\hline
$0, 0, 0, 0$ {\hglue 5pt}	   &    $a-b+c-d$ {\hglue 5pt}				             & $0,0,0$ \\
$0, 0, 0, 1$ 	   &    $a+b+c+d$						             & $0,0,1$ \\
$0, 0, 1, 1$ 	   &    $a+b+c-d$						             & $0,1,1$ \\    
$0, 1, 1, 1$ 	   &    $a      $						             & $1,1,1$ \\
$1, 1, 1, 1$ 	   &    $b      $						             & $1,1,1$ \\
$1, 1, 1, 2$ 	   &    $a+b-c+d$						             & $1,1,2$ \\    
$1, 1, 2, 2$ 	   &    $d      $						             & $1,2,2$ \\    
$1, 2, 2, 2$ 	   &    $a-b+c+d$						             & $1,2,2$ \\
$1, 2, 2, 3$ 	   &    $a+b-c-d$						             & $1,2,3$ \\        
$1, 2, 3, 3$ 	   &    $c      $						             & $1,3,3$ \\
$1, 3, 3, 3$ 	   &                                     &         
\\\hline                                
\end{tabular}\\
\bigskip
\caption{Induction Table for $\CA \in \CIF\setminus \HIF$.} 
\label{indtable:non-HIF} 
\end{table}

Let  $H_d := \ker d \in \CA$.
Next we show that the restriction $\CB := \CA^{H_d}$ is not free.
The defining polynomial of $\CB$ is $Q_\CB =   abc(a-b+c)(a+b+c)(a+b-c)$.
Fix $H_0 = \ker c \in \CB$ and consider the triple
$(\CB, \CB', \CB'')$.
One checks that $\CB'$ is free with $\exp \CB' = \{1, 2, 2\}$.
Now if $\CB$ were free, it would follow from 
\cite[Cor.\ 4.47]{orlikterao:arrangements}
that $\CB''$ is free with $\exp \CB'' \subset \exp \CB'$.
However, one checks that 
$\exp \CB'' = \{1, 3\} \not\subset \exp \CB'$.
Consequently, $\CB$ is not free and 
thus $\CA$ is not hereditarily inductively free.

In particular, this example also provides an easy
counterexample to Orlik's conjecture from 1981
that every free arrangement is hereditarily free, cf.\ 
\cite[Ex.\ 4.141]{orlikterao:arrangements}.
\end{example}

\subsection{Reflection Arrangements}
\label{ssect:refl}
The irreducible finite complex reflection groups were 
classified by Shephard and Todd, \cite{shephardtodd}.
Suppose that $W \subseteq \GL(V)$ 
is a finite, complex reflection group acting on the complex
vector space $V=\BBC^\ell$.
The \emph{reflection arrangement} $\CA = \CA(W)$ of $W$ in $V$ is 
the hyperplane arrangement 
consisting of the reflecting hyperplanes of the elements in $W$
acting as reflections on $V$.

\begin{remark}
\label{rem:q-ref}
Terao \cite{terao:freeI} has shown that 
every reflection
arrangement $\CA = \CA(W)$ is free, see also
\cite[Prop.\ 6.59]{orlikterao:arrangements}.
Moreover, by
\cite[\S 6, App.~C]{orlikterao:arrangements}
and \cite{hogeroehrle:free},  
every restriction 
$\CA^H$ is also again free (for any choice of hyperplane $H$).
It thus follows for any triple $(\CA,\CA', \CA'')$  
that  $\CA'$ fails to be free provided 
$\exp \CA'' \not\subseteq \exp \CA$,
by Corollary \ref{cor:q}, and so 
$\CA$ is not inductively free if $W$ is transitive on $\CA$.
This argument was used in \cite{hogeroehrle:free}
to show that $\CA(G_{33})$ and  $\CA(G_{34})$ are not inductively free.

More generally, by the explicit data
on exponents of restrictions $\CA^H$ %in the tables in 
for reflection arrangements $\CA$
in \cite[\S 6, App.\ C]{orlikterao:arrangements}, 
one readily checks that  
either $\exp \CA^H \subseteq \exp \CA$
for \emph{every} $H \in \CA$ or this 
containment fails for any $H \in \CA$.
\end{remark}

We summarize the discussion from Remark \ref{rem:q-ref} 
in our next result
which provides a very useful criterion for showing that 
a given reflection arrangement is not inductively free.

\begin{corollary}
\label{cor:nonindfree}
Let $\CA$ be a reflection arrangement. 
If $\exp \CA'' \not\subseteq \exp \CA$, then  
$\CA$ is not inductively free. 
\end{corollary}

\begin{example}
\label{ex:g333}
Let $\CA$ be the reflection arrangement 
of the monomial group $G(3,3,3)$.
Then by \cite[Prop.\ 6.82, Prop.\ 6.85, Cor.\ 6.86]{orlikterao:arrangements},
we have 
$\exp \CA'' = \{1,3\} \not\subseteq \exp \CA = \{1,4,4\}$.
By Corollary \ref{cor:nonindfree},
$\CA$ is not inductively free.
In view of Theorem \ref{main}, this is the smallest
example of a reflection arrangement that is not inductively free.
With some additional work one can show that 
this is the smallest 
example of a free arrangement  in dimension $3$ that
fails to be inductively free; here $|\CA| = 9$.
One can check 
that every proper free subarrangement of $\CA$ is also inductively free. 
In \cite[Ex.\ 4.1]{ziegler}, Ziegler
gave a similar example of a free $3$-arrangement $\CB$ with 
$9$ hyperplanes which is not inductively free. 
One can show that the lattice of $\CB$ coindices with $L(\CA)$. 
\end{example}

We can extend Lemma \ref{lem:3-arr} 
for reflection arrangements  to dimension $4$ 
as follows.

\begin{lemma}
\label{lem:4-arr}
Suppose $\ell = 4$, 
$\CA = \CA(W)$ is a 
reflection arrangement,
and $W$ is transitive on $\CA$.
Then  $\CA$ is inductively free if and only if $\CA$ is 
 hereditarily inductively free.
\end{lemma}

\begin{proof}
The reverse implication is clear.
So assume that $\CA \in \CIF$. 
By Definition \ref{def:indfree}, there is an $H_0 \in \CA$, so 
that $\CA^{H_0}$ is inductively free. 
But as $W$ is transitive on $\CA$, it follows that 
$\CA^H$ is inductively free for any $H \in \CA$.
Let $V \ne X \in L(\CA)\setminus \CA$.
Then $\CA^X$ is a $d$-arrangement  for $d \le 2$ and 
so $\CA^X$ is inductively free, by
Example \ref{ex:1-arr} and Lemma \ref{lem:2-arr}.
\end{proof}

\section{Proofs of Theorems \ref{main} and \ref{main2}}
\label{sect:proof}

\subsection{Proof of Theorem \ref{main}.}
\label{ss:pfmain}

Thanks to 
Proposition \ref{prop:product-indfree},
the question of inductive freeness reduces 
to the case when $\CA = \CA(W)$ is irreducible.
We prove Theorem \ref{main} by considering the different 
irreducible types of $W$ in turn, \cite{shephardtodd}.

\subsubsection{Cyclic groups}
By Example \ref{ex:1-arr}, 
$\CA(W)$ is inductively free
for $W$ a cyclic group.

\subsubsection{Coxeter groups}
\label{ssect:coxeter}
By \cite[Ex.\ 4.55]{orlikterao:arrangements},
the braid arrangement $\CA(S_\ell)$ is inductively free.
It was shown by Orlik, Solomon and Terao, that 
the reflection arrangements of the Coxeter groups of type $B_\ell$
for $\ell \ge 2$  and of type $D_\ell$
for $\ell \ge 4$ 
are also inductively free, \cite[Ex.\ 2.6]{jamboterao:free}. 
Barakat and Cuntz \cite{cuntz:indfree} 
completed this list by showing 
that every Coxeter arrangement is inductively free.

\subsubsection{Monomial groups}
Note that 
the reflection arrangements for 
$G(r,1,\ell)$, and  $G(r,p,\ell)$, 
for $r \ge 2$, $p \ne r$ and $\ell \ge 2$
are identical.
Thus we only consider $G(r,1,\ell)$ for $r \ge 2$ here.
Also note that $G(2,1,\ell)$ is the Coxeter group 
of type $B_\ell$ which is covered in \S \ref{ssect:coxeter} above, 
so we may assume that $r \ge 3$.

\begin{proposition}
\label{prop:gr1l}
Let $W = G(r,1,\ell)$ for $r\ge 3$ and $\ell \ge 2$.
Then $\CA(W)$ is inductively free.
\end{proposition}

\begin{proof}
Let $\CA_\ell(r) := \CA(W)$. 
Thanks to \cite[Prop.\ 6.77]{orlikterao:arrangements}, 
$\CA_\ell(r)$ is free with exponents
\[
\exp \CA_\ell(r) = \{1, r+1, 2r+1, \ldots, (\ell-1)r+1\}.
\]
We argue by induction on $\ell$.
For $\ell = 2$, the result follows from Lemma \ref{lem:2-arr}.
So we may assume that $\ell \ge 3$ and that 
$\CA_{\ell}(r)$ is inductively free.
Thus, by Proposition \ref{prop:product-indfree},
we see that the subarrangement 
$\CA_\ell(r) \times \Phi_1$ of $\CA_{\ell+1}(r)$
is inductively free with exponents
$\{\exp \CA_\ell(r), 0\}$.
We aim to show by induction of hyperplanes 
that $\CA_{\ell+1}(r)$ is inductively free,
Remark \ref{rem:indtable}.

Recall that the defining polynomial of $\CA_{\ell}(r)$ is given by
\[
Q_\ell := x_1 x_2 \cdots x_\ell 
\prod_{1\le i < j \le \ell} (x_i^r - x_j^r) = 
x_1 x_2 \cdots x_\ell \prod_{1\le i < j \le \ell} 
\left(\prod_{m=0}^{r-1} (x_i - \zeta^m x_j)\right),
\]
where $\zeta = e^{2\pi i/r}$ 
is a primitive $r$-th root of unity, 
see \cite[Ex.\ 6.29]{orlikterao:arrangements}.

So we start our induction of hyperplanes procedure
with the inductively free subarrangement
$\CA_\ell(r) \times \Phi_1$ of $\CA_{\ell+1}(r)$
with defining polynomial $Q_\ell$. 
We then add the hyperplanes $H_{\ell+1} :=\ker x_{\ell+1}$ and 
subsequently 
$H_{i,\ell+1}(m) := \ker(x_i - \zeta^m x_{\ell+1})$, 
for $1\le i \le \ell$ and $0 \le m < r$.
The crucial observation is that  at each stage
of this process 
the restriction is identical with 
$\CA_\ell(r)$ independent of $i$ and $m$.
The latter is again inductively free by induction.

The additional factors (other than the ones in $Q_\ell$) 
of the defining polynomial $Q_{\ell+1}$ 
of $\CA_{\ell+1}(r)$ are
$P := \{ x_{\ell+1}, x_i - \zeta^m x_{\ell+1}
\mid 1 \le i \le \ell, 0 \le m < r\}$.
The intermediate arrangements described above by adding the various 
hyperplanes $H_{\ell+1}$ and $H_{i,\ell+1}(m)$ to 
$\CA_\ell(r) \times \Phi_1$ have defining polynomial 
which is a product of $Q_\ell$
along with some factors from $P$.
Since the term 
$x_{\ell+1}$ occurs in each element of $P$, 
the restriction of any of these intermediate 
arrangements to 
$H_{\ell+1}$ or any $H_{i,\ell+1}(m)$
is achieved by substitution of $x_{\ell+1}$. 
The terms in $Q_\ell$ are 
not affected by these substitutions, because $x_{\ell+1}$ 
does not occur in $Q_\ell$.
We distinguish two different types of restrictions. 
The first one is the restriction of 
$\CA_\ell(r) \times \Phi_1$ to $H_{\ell+1}$, i.e.\ here we
replace $x_{\ell+1}$ by $0$. 
It follows that 
$(\CA_\ell(r) \times \Phi_1)^{H_{\ell+1}} \cong \CA_\ell(r)$.
A restriction of an intermediate subarrangement to 
any $H_{j,\ell+1}(m)$ results in the substitution 
$x_{\ell + 1} = \zeta^{-m} x_j$.
Therefore,  we get
$\{ \zeta^{-m}x_j, x_i - \zeta^{m'} x_j \mid 
1 \le i \le \ell, 0 \le m' < r\}$ as defining terms
for the restriction. But up to a scalar (and $0$), 
each such already occurs in $Q_\ell$. 
The zero does occur here, since
we restrict to the according coordinate hyperplane.
As as result the restriction to 
$H_{j,\ell+1}(m)$ of the intermediate arrangement is
again isomorphic to $\CA_{\ell}(r)$ which is 
inductively free by hypothesis.

We present the resulting induction
table for $\CA_{\ell+1}(r)$
(starting with $\CA_\ell(r) \times \Phi_1$)
in Table \ref{indtable1}.

\begin{table}[ht!b]\small
\renewcommand{\arraystretch}{1.5}
\begin{tabular}{lll}
  \hline
  $\exp \CA'$ & $\alpha_H$ & $\exp \CA''$\\
  \hline
  \hline
$\exp \CA_\ell,0$ & $x_{\ell+1}$ & $\exp \CA_\ell$ \\
$\exp \CA_\ell,1$ & $ x_1 - x_{\ell+1}$ & $\exp \CA_\ell$ \\
$\exp \CA_\ell,2$ & $x_1 - \zeta x_{\ell+1}$ & $\exp \CA_\ell$ \\
$\vdots$ & $\vdots$ &$\vdots$ \\
$\exp \CA_\ell,r$ & $x_1 - \zeta^{r-1} x_{\ell+1}$ & $\exp \CA_\ell$ \\
$\exp \CA_\ell,r+1$ & $x_2 - x_{\ell+1}$ & $\exp \CA_\ell$ \\
$\vdots$ & $\vdots$ &$\vdots$ \\
$\exp \CA_\ell,2r$ & $x_2 - \zeta^{r-1} x_{\ell+1}$ & $\exp \CA_\ell$ \\
$\vdots$ & $\vdots$ &$\vdots$ \\
$\exp \CA_\ell,(\ell-1)r$  & $x_{\ell-1} - \zeta^{r-1}x_{\ell+1}${\hglue 10pt}  & $\exp \CA_\ell$ \\
$\exp \CA_\ell,(\ell-1)r+1$ {\hglue 10pt}  & $x_{\ell} - x_{\ell+1}$ & $\exp \CA_\ell$  \\
$\vdots$ & $\vdots$ &$\vdots$ \\
$\exp \CA_\ell,\ell r $  & $x_{\ell} - \zeta^{r-1}x_{\ell+1}$ & $\exp \CA_\ell$ \\
$\exp \CA_\ell,\ell r+1$\\ 
\hline
\end{tabular}
\medskip
\caption{Induction Table for $\CA_{\ell+1}(r) = \CA(G(r,1,\ell+1))$.} \label{indtable1} 
\end{table}

The result thus follows from Definition \ref{def:indfree}
and the data in Table \ref{indtable1}.
\end{proof}

Now let $W = G(r,r,\ell)$ for $r, \ell \ge 2$.
If $r =2$, then $W$ is the Coxeter group of 
type $D_\ell$ and if $\ell =2$, then 
$W$ is a dihedral group.
So, both cases are covered in Section \ref{ssect:coxeter} above.
Thus we may assume that $r, \ell \ge 3$.

\begin{proposition}
\label{prop:grrl}
Let $W = G(r,r,\ell)$ for $r, \ell \ge 3$.
Then $\CA = \CA(W)$ is not inductively free.
\end{proposition}

\begin{proof}
By \cite[Cor.\ 6.86]{orlikterao:arrangements}, we have
\[
\exp \CA =  \{1, r+1, \ldots, (\ell-2)r+1, (\ell-1)(r-1)\},
\] 
and thanks to \cite[Prop.\ 6.82, Prop.\ 6.85]{orlikterao:arrangements}, we get 
\[
\exp \CA'' = \{1, r+1, \ldots, (\ell-3)r+1, (\ell-2)(r-1)+1\}.
\]
Since $r, \ell \ge 3$, $\exp \CA'' \not\subseteq \exp \CA$, and so 
$\CA$ is not inductively free, 
by  Corollary \ref{cor:nonindfree}.
\end{proof}

\subsubsection{Exceptional  groups (non-real)}
It follows from Lemma \ref{lem:2-arr} that 
the reflection arrangement of 
each of the rank $2$ groups 
of exceptional type is inductively free. 
The fact that $\CA(G_{25})$ is inductively free was checked in 
\cite[Ex.\ 6.92]{orlikterao:arrangements}.

{}From the exponents listed in the tables in   
\cite[App.\ C]{orlikterao:arrangements}
we infer that if $W$ is one of 
$G_{24}, G_{27}, G_{29}, G_{33}$ or $G_{34}$, then 
$\exp \CA'' \not\subseteq \exp \CA$,
for $\CA = \CA(W)$.
Note that by \cite{hogeroehrle:free},  
$\CA''$ is known to be free also for $G_{33}$ and $G_{34}$
with exponents given as in \cite[Tables C.14, C.17]{orlikterao:arrangements}.
Thus
$\CA$ is not inductively free in each of these instances, by Corollary \ref{cor:nonindfree}.

We investigate directly whether $\CA = \CA(W)$ is 
inductively free
in the three remaining exceptional cases $G_{26}, G_{31}$ and $G_{32}$.
It turns out that while $\CA(G_{26})$ and $\CA(G_{32})$ are inductively free
(here we present the induction tables),
in contrast, $\CA(G_{31})$ is not. % inductively free.

\begin{lemma}
\label{lem:g26}
Let $W = G_{26}$. Then $\CA(W)$ is inductively free.
\end{lemma}

\begin{proof}
Let $\zeta = e^{2\pi i/3}$ be a primitive $3$rd root of unity. 
We label the indeterminates 
of $S$ by $a, b$ and $c$.
The induction table for $\CA = \CA(W)$ is given in Table \ref{indtable:g26}.
Since $\CA$ is a $3$-arrangement, each restriction 
$\CA''$ is inductively free, by Lemma \ref{lem:2-arr}.
The result follows from Theorem \ref{thm:add-del}.
\end{proof}

\begin{table}[h]
 \extrarowheight3pt
\begin{tabular}[t]{lll}  \hline
  $\exp \CA'$ & $\alpha_H$ & $\exp \CA''$\\
%    [5pt]  
\hline\hline
$0 , 0 , 0$ & $b-c$ & $0,0 $ \\
$0 , 0 , 1$ & $c$ & $0,1$ \\
$0 , 1 , 1$ & $a+b+c$ & $1,1$ \\
$1 , 1 , 1$ & $b$ & $1,1$ \\
$1 , 1 , 2$ & $b -\zeta c$ & $1,1$ \\
$1 , 1 , 3$ & $a+b+ \zeta c$ & $1,3$ \\
$1 , 2 , 3$ & $a+b + \zeta^2 c$ & $1,3$ \\
$1 , 3 , 3$ & $b -\zeta^2 c$ & $1,3$ \\
$1 , 3 , 4$ & $a -\zeta b$ & $1,4$ \\
$1 , 4 , 4$ & $a -\zeta^2 b$ & $1,4$ \\
$1 , 4 , 5$ & $a -\zeta^2 c$ & $1,5$ \\
%\bottomrule
\hline 
\end{tabular}
  \qquad
\begin{tabular}[t]{lll}\hline
  $\exp \CA'$ & $\alpha_H$ & $\exp \CA''$\\ 
    \hline\hline 
$1 , 5 , 5$ & $a -\zeta c$ & $1,5$ \\
$1 , 5 , 6$ & $a + \zeta^2 b+c$ & $1,6$ \\
$1 , 6 , 6$ & $a+ \zeta b+c$ & $1,6$ \\
$1 , 6 , 7$ & $a +\zeta^2 b+ \zeta c$ & $1,7$ \\
$1 , 7 , 7$ & $a+ \zeta b+ \zeta c$ & $1,7$ \\
$1 , 7 , 8$ & $a + \zeta^2 b + \zeta^2 c$ & $1,7$ \\
$1 , 7 , 9$ & $a+ \zeta b +\zeta^2 c$ & $1,7$ \\
$1 , 7 , 10$ & $a-c$ & $1,7$ \\
$1 , 7 , 11$ & $a$ & $1,7$ \\
$1 , 7 , 12$ & $a-b$ & $1,7$ \\
$1 , 7 , 13$  \\
%\bottomrule
\hline 
\end{tabular}
\bigskip
\caption{Induction Table for $G_{26}$.} 
\label{indtable:g26} 
\end{table}

\begin{lemma}
\label{lem:g32}
Let $W = G_{32}$. Then $\CA(W)$ is inductively free.
\end{lemma}

\begin{proof}
Let $\zeta = e^{2\pi i/3}$ be a primitive $3$rd root of unity. 
We label the indeterminates 
of $S$ by $a, b, c$ and $d$.
We present the induction table for $\CA = \CA(W)$ in 
Table \ref{indtableg32} below.
Here at each step the restriction $\CA''$ is a 
$3$-arrangement. We checked in each case
that $\CA''$ is indeed itself again inductively free.
One easily checks from the data given that at each step
$\exp \CA'' \subset \exp \CA'$.

\begin{table}[h] %\small %[ht!b] 
\extrarowheight3pt
\begin{tabular}[t]{lll}  \hline
  $\exp \CA'$ & $\alpha_H$ & $\exp \CA''$\\
%    [5pt]  
\hline\hline
$0, 0, 0, 0$ 	&    $c$						& $0,0,0$ \\ 
$0, 0, 0, 1$     &    $a+b+c$          			& $0, 0,  1$ \\
$0, 0, 1, 1$     &    $b$                				&     $0, 1,  1$ \\
$0, 1, 1, 1$     &    $a-b-d$           				& $1, 1,  1$ \\
$1, 1, 1, 1$     &     $a+b+\zeta c$ 			& $1, 1,  1$ \\
$1, 1, 1, 2$     &     $a+b+\zeta^2 c$   		&   $1, 1,  1$ \\
$1, 1, 1, 3$     &     $a+\zeta^2 b+c$        		& $1, 1,  3$\\ 
$1, 1, 2, 3$     &  $a+\zeta b  +c$              		& $ 1, 1,  3$ \\     
$1, 1, 3, 3$     & $a -\zeta  b  -d$              	& $1, 3,  3$ \\ 
$1, 2, 3, 3$     & $a+\zeta b  +\zeta^2 c$   	&   $1, 2,  3$ \\
$1, 2, 3, 4$     & $a+\zeta b  +\zeta c$   		&     $1, 2,  4$ \\ 
$1, 2, 4, 4$     & $a -\zeta^2 b  -d$           	&  $1, 4,  4$\\
$1, 3, 4, 4$     & $a -\zeta  c + \zeta^2 d$ 	& $1, 4,  4$ \\ 
$1, 4, 4, 4$     & $a+ \zeta^2 b  + \zeta c$   	&   $ 1, 4,  4$ \\ 
$1, 4, 4, 5$     & $a - \zeta^2 c + \zeta^2 d$ 	& $ 1, 4,  5$ \\ 
$1, 4, 5, 5$     & $b - c -\zeta^2 d$ 			& $1, 5,  5$ \\ 
$1, 5, 5, 5$     & $a+ \zeta^2 b  + \zeta^2 c$   &   $ 1, 5,  5$ \\ 
$1, 5, 5, 6$     & $a-c+ \zeta^2 d$ 			& $ 1, 5,  6$ \\
$1, 5, 6, 6$     & $b - \zeta^2 c - \zeta^2 d$ 	& $   1, 6,  6$ \\
$1, 6, 6, 6$     &  $   b -\zeta c - \zeta^2 d$	& $    1, 6,  6$ \\
%\bottomrule
\hline 
\end{tabular}
  \qquad
\begin{tabular}[t]{lll}\hline
  $\exp \CA'$ & $\alpha_H$ & $\exp \CA''$\\ 
    \hline\hline 
$1, 6, 6, 7$ 	& $a$ 						& $ 1, 6,  7$ \\ 
$1, 6, 7, 7$     & $a - \zeta^2 c +\zeta d$ 		& $     1, 7,  7$ \\ 
$1, 7, 7, 7$     &  $a-c+\zeta d$ 				& $     1, 7,  7$ \\
$1, 7, 7, 8$     &  $b - \zeta c  -\zeta d$ 		& $     1, 7,  8$ \\ 
$1, 7, 8, 8$     & $a - \zeta c +\zeta d$ 		& $      1, 7,  8$ \\
$1, 7, 8, 9$     & $b - \zeta^2 c  - \zeta d$ 	& $    1, 7,  9$ \\ 
$1, 7, 9, 9$     & $b-c - \zeta d$ 				& $ 1, 7,  9 $ \\
$1, 7, 9, 10$   & $a - \zeta b  - \zeta d$ 	& $     1, 7,  9$ \\
$1, 7, 9, 11$    & $a - \zeta^2 b -\zeta^2 d$ & $   1, 7,  9$ \\
$1, 7, 9, 12$    & $b-c-d$              			& $   1, 7, 12$ \\
$1, 7, 10, 12$   & $b - \zeta c -d$ 			& $ 1, 7, 12$ \\
$1, 7, 11, 12$   & $b -\zeta^2 c -d$           	& $ 1, 7, 12$ \\
$1, 7, 12, 12$   & $a - \zeta^2 b  - \zeta d$ 	& $    1, 7, 12$ \\
$1, 7, 12, 13$   & $a-b - \zeta^2 d$ 			& $    1, 7, 13$ \\
$1, 7, 13, 13$   & $a-b - \zeta d$ 			& $     1, 7, 13 $ \\
$1, 7, 13, 14$   & $a - \zeta b  -\zeta^2 d$ 	& $    1, 7, 13$ \\
$1, 7, 13, 15$    & $a - \zeta c +d$        		& $      1, 7, 13$ \\
$1, 7, 13, 16$   & $ d$                       			& $    1, 7, 13$ \\ 
$1, 7, 13, 17$   & $a -\zeta^2 c +d $           	& $ 1, 7, 13$\\ 
$1, 7, 13, 18$   & $a-c+d$            			& $   1, 7, 13$ \\ 
$1, 7, 13, 19$   \\ 
\hline
\end{tabular}
\bigskip
\caption{Induction Table for $G_{32}$.} \label{indtableg32} 
\end{table}

In case $\exp \CA'' = \{1,7,13\}$, the restriction 
$\CA''$ is always isomorphic to the reflection 
arrangement $\CA(G_{26})$ of $G_{26}$, which is inductively free, by
Lemma \ref{lem:g26}.

Thus, by Definition \ref{def:indfree}, Theorem \ref{thm:add-del}
and the data from Table \ref{indtableg32},
%it follows by induction that 
$\CA$ is inductively free.
\end{proof}

\begin{lemma}
\label{lem:g31}
Let $W = G_{31}$. Then $\CA(W)$ is not inductively free.
\end{lemma}

\begin{proof}
Let $\CA = \CA(W)$ be the reflection arrangement of $W$.
Then $\exp \CA = \{1, 13, 17, 29\}$,
\cite[Table C.12]{orlikterao:arrangements}.
So $|\CA| = 60$.
Suppose that $\CA$ is inductively free so it has an induction table. 
Since all hyperplanes in $\CA$ are conjugate and since 
$\CA''$ is free with exponents given by $\exp \CA '' = \{1, 13, 17\}$,
we have $|\CA''| = 31$. Clearly, this is the maximal cardinality of 
any restriction in the induction table of $\CA$. 

We study the induction table of $\CA$ from its end
rather than its beginning.
By carefully analyzing the possibilities of the 
occurring free subarrangements we
are able to deduce a contradiction
to our assumption that $\CA$ is inductively free.

First, since $W$ is transitive on $\CA$, we may 
assume that the induction table ends with the addition of a fixed
hyperplane.
Going back in the induction table, we
can remove $13$ hyperplanes 
from $\CA$ where at each stage the restriction  
is inductively free with the same set of exponents $\{1, 13, 17\}$.
The reason for that stems from the fact that
the exponents of the restriction in each step have to be a subset of 
$\{1, 13, 17, b\}$, where $b \ge 17$. 
But we have already seen that the maximal cardinality of
such a restriction is $31$.
 
This results in a free subarrangement of $\CA$ with $47$ hyperplanes.

We then construct 
all free subarrangements $\CC$ of $\CA$ with 
$47$ hyperplanes 
(there are roughly $100.000$ of them).
Then we check that any restriction  $\CC''$ to a hyperplane of 
any such subarrangement $\CC$, allowing us to extend our induction table further back, 
results again in a $3$-arrangement with 
exponents $\exp \CC'' = \{1, 13, 17\}$.
We can continue to remove hyperplanes 
while maintaining the same set of admissible exponents on 
the resulting restrictions until we 
arrive at a subarrangement, $\CB$ say, with $40$ 
hyperplanes.
It turns out that 
we necessarily have to have $\exp \CB = \{1,9,13,17\}$ and 
there are only two such subarrangements $\CB$ 
such that we obtain a valid induction from $\CB$ to all of $\CA$.
Now one can check that every  restriction $\CB''$ of $\CB$ to a hyperplane admits $21$ hyperplanes
(in both remaining instances for $\CB$).
While each of the $3$-arrangements $\CB''$ is still free,
it follows from Theorem \ref{thm:add-del} that
the corresponding subarrangement $\CB'$ is not free,
as $21$ is not realized as a triple sum of
$\exp \CB = \{1,9,13,17\}$.
Consequently, 
as our induction table necessarily does have to pass through one of only two possible 
choices of a free subarrangement $\CB$ with $40$ hyperplanes, and as $\CB'$ is
not free in any case, $\CA$ is not inductively free and we get a contradiction. 
Thus, $\CA$ is not inductively free, as claimed.
\end{proof}

This completes the proof of Theorem \ref{main}.

\begin{remark}
\label{rem:computations}
In order to establish the results of Lemmas \ref{lem:g26} to \ref{lem:g31}, 
we first use the functionality for complex reflection groups 
provided by the   \CHEVIE\ package in   \GAP\ 
(and some \GAP\ code by J.~Michel)
(see \cite{gap3} and \cite{chevie})
in order to obtain explicit 
linear functionals $\alpha$ defining the hyperplanes 
$H = \ker \alpha$ of the reflection arrangement
$\CA(W)$. 
These then allow us to subsequently implement the 
module of derivations $D(\alpha)$ associated with $\alpha$
in the   \Singular\ computer algebra system (cf.~\cite{singular}). 
We then use the module theoretic functionality of
  \Singular\ to show that the
modules of derivations in question are 
free and ultimately are able to show that 
in case of $G_{26}$ and $G_{32}$
the arrangement is inductively free
for a suitable chain of subarrangements obeying 
Definition \ref{def:indfree}.

In Lemma \ref{lem:g31} we use in addition 
the functionality of \Sage\ to compute the intersection lattice of 
$\CA(G_{31})$ 
and then to construct the candidates of an induction table 
for $\CA(G_{31})$, \cite{sage}.
\end{remark}

\subsection{Proof of Theorem \ref{main2}.}
\label{ss:pfmain2}

In view of Corollary \ref{cor:product-heredindfree}, 
Theorem \ref{main2}
follows once we have shown that whenever
$W$ is irreducible and 
$\CA(W)$ is inductively free, that then $\CA(W)$ is  
hereditarily inductively free.
We prove this again by considering the different 
irreducible types of $W$ in turn, \cite{shephardtodd}.

\subsubsection{Cyclic groups}
In case $W$ is a cyclic group, this follows
from Example \ref{ex:1-arr}.

\subsubsection{Coxeter groups}
In \cite[Cor.\ 5.15]{cuntz:indfree}, Barakat and Cuntz showed 
that every crystallographic arrangement is 
hereditarily inductively free.
This covers all cases for $W$ a Weyl group.
In \cite[\S 5.4]{cuntz:indfree}, the authors showed 
that both  $\CA(H_3)$ and $\CA(H_4)$ are 
inductively free. It thus follows from 
Lemmas \ref{lem:3-arr} and \ref{lem:4-arr}
that $\CA(H_3)$ and $\CA(H_4)$ are also
hereditarily inductively free.

\subsubsection{Monomial groups}
It suffices to consider $W = G(r,1,\ell)$ for $r\ge 3$ and $\ell \ge 2$.
Let $\CA = \CA_\ell(r) = \CA(W)$ and let $X \in L(\CA)$.
Thanks to \cite[Prop.\ 6.77]{orlikterao:arrangements},
$\CA^X$ is isomorphic to $\CA_p(r) = \CA(G(r,1,p))$,
where  $p = \dim X$.
Thus, it follows from 
Proposition \ref{prop:gr1l}
that $\CA^X$ is inductively free.
For $r, \ell \ge 3$, the arrangement
$\CA(G(r,r,\ell))$ is not inductively free, by 
Proposition \ref{prop:grrl}.

\subsubsection{Exceptional  groups (non-real)}
Now let $W$ be a non-real, irreducible, exceptional 
complex reflection group.
If $\ell = 2$, then $\CA(W)$ is 
hereditarily inductively free, thanks to 
Lemma \ref{lem:2-arr}.
If $\ell = 3$ and $\CA(W)$ is 
inductively free, then $\CA(W)$ is 
hereditarily inductively free, 
by Lemma \ref{lem:3-arr}. 
If $\ell = 4$ and $\CA(W)$ is 
inductively free, then 
$W$ is transitive on $\CA$, by Theorem \ref{main}, and so,
by Lemma \ref{lem:4-arr}, 
$\CA(W)$ is 
hereditarily inductively free.
If $\ell > 4$, then $\CA(W)$ is 
not inductively free, 
by Theorem \ref{main}.

This completes the proof of Theorem \ref{main2}.

%%%%%%%%%%%%%%%%%%%%%%%%%%%%%%%%%%%%%%%%%%%%%%%%%%%%%%%%%%%%%%%%%%%%%%
%%%%%%%%%%%%% Acknowledgments
%%%%%%%%%%%%%%%%%%%%%%%%%%%%%%%%%%%%%%%%%%%%%%%%%%%%%%%%%%%%%%%%%%%%%%

%\bigskip {\bf Acknowledgments}: 

%%%%%%%%%%%%%%%%%%%%%%%%%%%%%%%%%%%%%%%%%%%%%%%%%%%%%%%%%%%%%%%%%%%%%%
%%%%%%%%%%%%% bibliography
%%%%%%%%%%%%%%%%%%%%%%%%%%%%%%%%%%%%%%%%%%%%%%%%%%%%%%%%%%%%%%%%%%%%%%

\bigskip

\bibliographystyle{amsalpha}
%%\bibliography{matts} 

\newcommand{\etalchar}[1]{$^{#1}$}
\providecommand{\bysame}{\leavevmode\hbox to3em{\hrulefill}\thinspace}
\providecommand{\MR}{\relax\ifhmode\unskip\space\fi MR }
\providecommand{\MRhref}[2]{%
  \href{http://www.ams.org/mathscinet-getitem?mr=#1}{#2} }
\providecommand{\href}[2]{#2}

%%%%%%%%%%%%%%%%%%%%%%%%%%%%%%%%%%%%%%%%%%%%%%%%%%%%%%%%%%%%%%%%%%%%%%
%%%%%%%%%%%%%%%%%%%%%%%%%%%%%%%%%%%%%%%%%%%%%%%%%%%%%%%%%%%%%%%%%%%%%%

\end{document}